\documentclass[11pt, final]{amsart}
\usepackage{amsmath, amsxtra, amssymb, mathdots}
\usepackage{verbatim}

\usepackage[svgnames]{xcolor}
\usepackage{tikz}
\usepgflibrary{decorations.text}
\usetikzlibrary{decorations.text}
\pgfdeclarelayer{edgelayer}
\pgfdeclarelayer{nodelayer}
\pgfsetlayers{edgelayer,nodelayer,main}

\tikzstyle{none}=[inner sep=0pt]
\definecolor{hexcolor0xfefdfd}{rgb}{0.996,0.992,0.992}

\usepackage{graphicx}
\usepackage[cmtip,all]{xy}
\usepackage{xypic}

\usepackage[notcite, notref]{showkeys}
\usepackage[colorlinks, linkcolor=black, citecolor=blue]{hyperref}

\theoremstyle{plain}
\newtheorem*{main-theorem}{Theorem}
\newtheorem{theorem}{Theorem}[section]
\newtheorem{prop}[theorem]{Proposition}

\newtheorem{claim}[theorem]{Claim}
\newtheorem{lemma}[theorem]{Lemma}

\theoremstyle{definition}
\newtheorem{definition}[theorem]{Definition}
\newtheorem{example}[theorem]{Example}

\newtheorem{construction}[theorem]{Construction}

\numberwithin{equation}{section}

\def\ra{\rightarrow}

\newcommand\arr{\ifinner\to\else\longrightarrow\fi}

\newcommand\Mg[1]{\overline{\mathcal{M}}_{#1}}
\newcommand\M{\overline{M}}

\def\co{\colon\thinspace} 

\DeclareMathOperator{\Nef}{Nef}
\DeclareMathOperator{\Pic}{Pic}

\DeclareMathOperator{\card}{card}

\def\Hn1{\mathcal{H}_{n,1}}

\renewcommand\S{\mathfrak{S}}

\def\M{\overline{M}}

\def\QQ{\mathbb{Q}}
\def\ZZ{\mathbb{Z}}

\def\ZZ{\mathbb{Z}}

\def\DD{\mathbb{D}}

\def\EE{\mathbb{E}}
\newcommand\modp[2]{\langle #1\rangle_{#2}}
\def\sl{\mathfrak{sl}}

\begin{document}
\title{New nef divisors on $\M_{0,n}$}
\author[M. Fedorchuk]{Maksym Fedorchuk}
\address{Department of Mathematics\\
Boston College \\
Carney Hall 324\\
140 Commonwealth Avenue\\
Chestnut Hill, MA 02467}
\email{maksym.fedorchuk@bc.edu}

\begin{abstract}
We give a direct proof, valid in arbitrary characteristic, 
of nefness for two families of F-nef divisors on $\M_{0,n}$.
The divisors we consider include all type A level one conformal block divisors
as well as divisors previously not known to be nef. 
\end{abstract}
\date{\today}
\maketitle

\setcounter{section}{0}
\setcounter{tocdepth}{1}
\tableofcontents

\section{Introduction}
\label{S:introduction}

We prove nefness for two families of divisors on $\M_{0,n}$ by a new method.
The first family $\mathcal{D}_1$ consists of all type A level one conformal block divisors 
and has many geometric incarnations; see Section \ref{S:D1}.
We give a new proof that every divisor in $\mathcal{D}_1$ defines a base-point-free linear system on $\M_{0,n}$,
which is an isomorphism on $M_{0,n}$. 

The conformal block divisors on $\M_{0,n}$ form an important family: 
They are nef by work of Fakhruddin \cite{fakh}, 
often span extremal rays of the (symmetric) nef cone \cite{agss,ags}, and are related
to Veronese quotients \cite{gian,gian-gib,GJM,GJMS}. 
In particular, CB divisors account for all known regular morphisms
on $\M_{0,n}$ \cite[Section 19.5]{conf-block}. 

As an application of our method, we construct a family $\mathcal{D}_2$
of nef divisors on $\M_{0,n}$, 
which do not lie in the cone spanned by the conformal block divisors.\footnote{We do not
know a way to prove that a given divisor is not an effective combination of CB divisors (but see \cite{swinarski} for some results
in this direction). Rather, 
there are explicit examples of divisors in $\mathcal{D}_2$ for which extensive numerical experimentation has failed to uncover CB 
divisors whose span could contain them.} The geometric meaning of this family is elusive, but we speculate that it is connected
to the morphisms defined by divisors in $\mathcal{D}_1$; see Section \ref{S:D2}.

Our proof is elementary in that it relies only on  Keel's relations in $\Pic(\M_{0,n})$ and nothing else. One
advantage of this approach is that it works in positive characteristic, where semiampleness of the conformal block
divisors on $\M_{0,n}$ is not generally known. 

The key observations that enable our proof are:
\begin{enumerate}
\item The family $\mathcal{D}_i$ is functorial with respect to the boundary stratification, 
or, equivalently, satisfies factorization in the sense of \cite{BG}. 
\item Every divisor $D\in \mathcal{D}_i$ is linearly equivalent to an effective combination of the boundary divisors on $\M_{0,n}$.
\end{enumerate}
A standard argument implies that all divisors in $\mathcal{D}_i$ are nef. 
(To prove semiampleness of divisors in $\mathcal{D}_1$, we show that every $D\in \mathcal{D}_1$ is linearly equivalent
to an effective combination of boundary divisors in such a way as to avoid any given point of $\M_{0,n}$.)

The above argument is at the heart of the original inductive approach to the F-conjecture for $\M_{0,n}$; see \cite[Question (0.13)]{GKM}
and the discussion surrounding it.
F-nef divisors obviously satisfy factorization and the strong F-conjecture says
that every F-nef divisor can be written as an effective combination of boundary divisors. 
However, a recent result of Pixton shows that the strong F-conjecture is false:
there exists a nef divisor on $\M_{0,12}$ which is not an effective combination of the boundary divisors \cite{pixton}.

Nevertheless, one could still hope that the strong F-conjecture holds for a restricted class of F-nef divisors;
restrictions of symmetric F-nef divisors to the boundary being one example. The divisors in $\mathcal{D}_i$ are of this form.
Thus the present paper can be regarded as further evidence for the symmetric F-conjecture.

It would be interesting to know whether all conformal block divisors on $\M_{0,n}$ are effective combinations of the boundary divisors
(see the discussion after Theorem \ref{T:T1} below).

\subsection{Notation}

The $i^{th}$ cotangent line bundle and its divisor class on $\M_{0,n}$ is denoted $\psi_i$. We
use the notation $[n]=\{1,\dots, n\}$
and say that a partition $I \sqcup J=[n]$ is \emph{proper} if $\vert I\vert, \vert J \vert \geq 2$.
The boundary divisor on $\M_{0,n}$ corresponding to a proper partition $I\sqcup J=[n]$ is denoted $\Delta_{I,J}$.
We write $\Delta=\sum_{I,J} \Delta_{I,J}$ for the total boundary divisor; here and elsewhere the summation 
is taken over all proper partitions of $[n]$, unless specified otherwise.

Given a set $S$, we denote by $\Gamma(S)$ the complete graph on $S$ and by $E(S)$ the set of all edges
of $\Gamma(S)$. A \emph{weighting} or a \emph{weight function} on $\Gamma(S)$ is a function $w\co E(S) \ra \QQ$. 
For the complete graph $\Gamma([n])$ on the set $[n]$, we write $(i-j)$ to denote the edge joining 
vertices $i$ and $j$.
Given a weight function $w$ on $\Gamma([n])$,
we make the following definitions:
\begin{enumerate}
\item The \emph{$w$-flow through a vertex} $k\in [n]$ is defined to be 
\[
w(k):=\sum_{i\neq k}(w(k-i)).
\]
\item The \emph{$w$-flow across a partition} 
$I\sqcup J=[n]$ is defined to be 
\[
w(I\mid J)=\sum_{i\in I, j\in J} w(i-j).
\] 
\end{enumerate}
A \emph{degree function} on $\Gamma(S)$ is a function $S \ra \ZZ$. Given a degree function $i\mapsto d_i$ on $\Gamma([n])$,
we set $d(I):=\sum_{i\in I} d_i$ for any $I\subset [n]$. Given an integer $m\geq 2$, we say that $I\subset [n]$ is \emph{$m$-divisible}
if $m\mid d(I)$. We call $I\sqcup J=[n]$ an \emph{$m$-partition} if $m\mid d(I)$ and $m\mid d(J)$. 

We denote by $\modp{a}{m}$ the representative in $\{0,1,\dots, m-1\}$ of $a$ modulo $m$.

We work over an algebraically closed field of arbitrary characteristic. 

\section{Weighted graphs and effective combinations of boundary}

Every divisor on $\M_{0,n}$ can be written as 
\begin{align*}
\sum_{i=1}^n a_i \psi_i - \sum_{I,J} b_{I,J} \Delta_{I,J}.
\end{align*}
This representation is far from unique because we have 
the following relation in $\Pic\left(\M_{0,n}\right)$ for every $i\neq j$:
\begin{equation}\label{E:relation}
\psi_i+\psi_j=\sum_{i\in I,j\in J}\Delta_{I,J}.
\end{equation}
Relations \eqref{E:relation} generate the module of all relations among $\{\psi_i\}_{i=1}^{n}$ and
$\{\Delta_{I,J}\}$; this follows,
for example, from \cite[Theorem 2.2(d)]{AC}, which in turn follows from Keel's relations \cite{keel}.
(We note that the above representation is unique if we impose an additional condition $\vert I\vert, \vert J \vert \geq 3$;
see \cite[Lemma 2]{FG}.)

We now state a simple observation that we will use repeatedly in the sequel.
\begin{lemma}[Effectivity criterion]\label{L:main} Let $R=\ZZ$ or $R=\QQ$.
A divisor $D=\sum_{i=1}^n a_i \psi_i - \sum_{I,J} b_{I,J} \Delta_{I,J}$ 
is $R$-linearly equivalent to $\sum_{I,J} c_{I,J}\Delta_{I,J}$ if and only if there is an $R$-valued weighting of 
$\Gamma([n])$ such that the flow through each vertex $i$ is $a_i$ and the
flow across each proper partition $I\sqcup J=[n]$ is $b_{I,J}+c_{I,J}$.

In particular, $D$ is an effective $R$-linear combination of the boundary divisors
on $\M_{0,n}$ if and only if there exists an $R$-valued weighting of $\Gamma([n])$ such
that the flow through each vertex $i$ is $a_i$ and the
flow across each proper partition $I\sqcup J=[n]$ is at least $b_{I,J}$.
\end{lemma}
\begin{proof}

Suppose that for each $i\neq j$ we use the relation \eqref{E:relation} $w(i-j)$ times to rewrite $D$ as 
$\sum_{I,J} c_{I,J} \Delta_{I,J}$.
Then in the free $R$-module generated by $\{\psi_i\}_{i=1}^n$ and $\{\Delta_{I,J}\}$ we have
\begin{multline*}
\sum_{I,J} c_{I,J} \Delta_{I,J}=D-\sum_{i\neq j} w(i-j)\left(\psi_i+\psi_j-\sum_{i\in I,j\in J}\Delta_{I,J}\right)
\\=\sum_{i=1}^n \bigl(a_i-w(i)\bigr) \psi_i -\sum_{I,J} \bigl(b_{I,J}-w(I\mid J)\bigr) \Delta_{I,J}.
\end{multline*}
The claim follows. 
\end{proof}

\section{Type A level one conformal block divisors revisited}
\label{S:D1}

In this section, we study the family $\mathcal{D}_1$ of type A level one conformal block divisors.

\begin{definition}\label{D:D}
Consider $n$ integers $(d_1,\dots, d_n)$ and
let $m\geq 2$ be an integer dividing $\sum_{i=1}^n d_i$.
We define a divisor on $\M_{0,n}$ by the following formula
\begin{equation}\label{E:D}
D\bigl((d_1,\dots, d_n), m\bigr)
=\sum_{i=1}^n \modp{d_i}{m}\modp{m-d_i}{m}\psi_i
-\sum_{I,J} \modp{d(I)}{m}\modp{d(J)}{m}\Delta_{I,J},
\end{equation}
where $d(I):=\sum_{i\in I} d_i$.
\end{definition}

The motivation for this definition comes from the following
observation \cite[Proposition 4.8]{fedorchuk-cyclic}:
\begin{prop}\label{P:new-formula}
For a weight vector $\vec{d}=(d_1,\dots, d_n)$, let $m$ be an integer dividing $\sum_{i=1}^n d_i$.
Let $\EE$ be the pullback to $\M_{0,n}$ of the Hodge bundle over $\Mg{g}$ via the weighted cyclic $m$-covering morphism $f_{\vec{d},m}$ 
and let $\EE_j$ be the eigenbundle of $\EE$ associated to the character $j$ of $\mu_m$. 
Then
\[
\det \EE_j =\frac{1}{2m^2}\left[\sum_{i=1}^n \modp{jd_i}{m}\modp{m-jd_i}{m}\psi_i
-\sum_{I,J} \modp{jd(I)}{m}\modp{jd(J)}{m}\Delta_{I,J}\right].
\]
\end{prop}

The divisor $D\bigl((d_1,\dots, d_n),m\bigr)$ has at least three incarnations:
\begin{enumerate}
\item It is a type A level $1$ conformal block divisor; see \cite{fakh} for the definition and \cite{agss} for a detailed study of these divisors. 
More precisely, the $\sl_m$ level $1$ conformal block divisor 
$\DD(\sl_m, 1, (d_1,\dots,d_n))$ equals to $2m^2 D\bigl((d_1,\dots, d_n),m\bigr)$ \cite{fedorchuk-cyclic}. 
\item It is a pullback of a natural polarization on a GIT quotient
of a parameter space of $n$-pointed rational normal curves
\cite{gian,gian-gib}. 
\item It is a determinant of a Hodge eigenbundle as in Proposition \ref{P:new-formula}.
\end{enumerate}
Each interpretation of $D\bigl((d_1,\dots,d_n),m\bigr)$ leads to an independent proof of its nefness.
The first via the theory of conformal blocks which 
realizes $D\bigl((d_1,\dots, d_n),m\bigr)$ as a quotient of a trivial 
vector bundle over $\M_{0,n}$; the second via GIT, and the third via the semipositivity of the Hodge
bundle over $\Mg{g}$, which in turn comes from the Hodge theory \cite{kollar-projectivity}.

We now propose a fourth proof of nefness of $D\bigl((d_1,\dots,d_n),m\bigr)$ which is independent of all of the above
and is completely elementary.

Since the definition of $D((d_1,\dots, d_n),m)$ depends only on $\modp{d_i}{m}$, 
replacing each $d_i$ by $\modp{d_i}{m}$ we can 
assume that $d_i\leq m-1$. 
We proceed with a construction of a certain weighting on the complete graph $\Gamma([n])$.

\begin{prop}[Standard construction]\label{P:standard}
Suppose $1\leq d_i\leq m-1$ and $m \mid \sum_{i=1}^n d_i$.
There exists a weighting $w$ of $\Gamma([n])$ such that
\begin{enumerate}
\item For every vertex $i\in \Gamma([n])$, we have $w(i)=d_i(m-d_i)$.
\item For every proper partition $I\sqcup J=[n]$, we have $w(I\mid J)\geq \modp{d(I)}{m}\modp{d(J)}{m}$. 
\end{enumerate}
In addition, 
\begin{enumerate}
\item[(3)] Given a curve $[C] \in \M_{0,n}$, we can choose $w$ so that 
$w(I\mid J)=\modp{d(I)}{m}\modp{d(J)}{m}$ for any $I\sqcup J=[n]$ satisfying $[C]\in \Delta_{I,J}$.
\item[(4)] For any fixed proper partition $I\sqcup J=[n]$, the weighting $w$ can be chosen so that $w(I\mid J)\geq 2m+\modp{d(I)}{m}\modp{d(J)}{m}$. 
\end{enumerate}
\end{prop}
\begin{proof}

Let $\sum_{i=1}^n d_i = ms$ and let $S=\{p_1, \dots, p_{ms}\}$ be a multiset of
indices where each index $i\in \{1,\dots, n\}$ appears $d_i$
times. 

Choose a cyclic permutation $\sigma\in \S_n$ of $[n]=\{1,\dots, n\}$ 
and arrange the elements of $S$ along a circle 
at the vertices of a regular $ms$-gon so that the $d_i$ occurrences
of each $i$ are adjacent, and the order of $\{1,\dots,n\}$ along the circle is given by $\sigma$. 

Define $\{S_k\}_{k=1}^m$ to be the regular
$s$-gons formed by the chords divisible by $m$. 
Since $d_i\leq m-1$, this subdivision satisfies the following 
property:
\begin{itemize}
\item Each $S_k$ contains at most one occurrence of each $i\in \{1,\dots, n\}$.
\end{itemize}
For every edge $e\in E(S)$, we define
\begin{align*}
w_1(e) &=\begin{cases} 1 & \text{if $e$ joins distinct indices}, \\
0 & \text{otherwise}.
\end{cases}
\\
w_2(e)&=\begin{cases} -m & \text{if $e\in E(S_k)$ for some $k=1,\dots, m$}, \\
0 & \text{otherwise}.
\end{cases}
\end{align*}

The weight function $w_1+w_2$ on $\Gamma(S)$ induces in an obvious way a weight 
function $w$ on $\Gamma([n])$. Namely, the $w$-weight of the edge $(i-j)$ in $\Gamma([n])$ is the sum of
the $(w_1+w_2)$-weights of all edges in $S$ joining the indices $i$ and $j$. (Note that by construction,
the $(w_1+w_2)$-weight of any edge in $S$ joining two equal indices is $0$).  

We now compute the $w$-flow through each vertex $i\in [n]$. Clearly, the contribution of $w_1$ to $w(i)$
is $d_i(ms-d_i)$ and the contribution of $w_2$ to $w(i)$ is $-md_i(s-1)$. Therefore,
\[
w(i)=d_i(ms-d_i)-md_i(s-1)=d_i(m-d_i).
\] This establishes (1). Next we show that
the $w$-flow across each proper partition $I\sqcup J=[n]$ is 
at least $\modp{d(I)}{m}\modp{d(J)}{m}$.

Recall that $d(I)=\sum_{i\in I} d_i$. 
Write $d(I)=mq+r$.
Let $x_1, \dots, x_m$ be the number of indices from $I$
occurring in each of the sets $S_1, \dots, S_m$.
Then $x_1+\cdots+x_m=d(I)=mq+r$. Tracing through the 
construction we see that $w_1(I\mid J)=d(I)(ms-d(I))$ and
$w_2(I\mid J)=-m \sum_{k=1}^m x_k(s-x_k)$. It follows that
\begin{equation}\label{E:quadratic}
w(I\mid J)=d(I)(ms-d(I))-m\sum_{k=1}^m x_k(s-x_k).
\end{equation}
Since $x(s-x)$ is a concave function of $x$, the minimum
in \eqref{E:quadratic} under the constraint $x_1+\cdots+x_m=mq+r$ is achieved when the $x_i$'s differ
by at most $1$ from each other, 
i.e. when $r$ of $x_i$'s are equal to $q+1$ and $m-r$ of $x_i$'s are equal
to $q$. (When this happens, we say that $I\sqcup J$ is balanced with respect to $\sigma$.) A straightforward computation now shows that
for these values of $x_i$'s Equation 
\eqref{E:quadratic} evaluates to 
\[
r(m-r)=\modp{d(I)}{m}\modp{d(J)}{m}.
\]
This finishes the proof of (2).
\begin{figure}
\scalebox{0.5}{\includegraphics{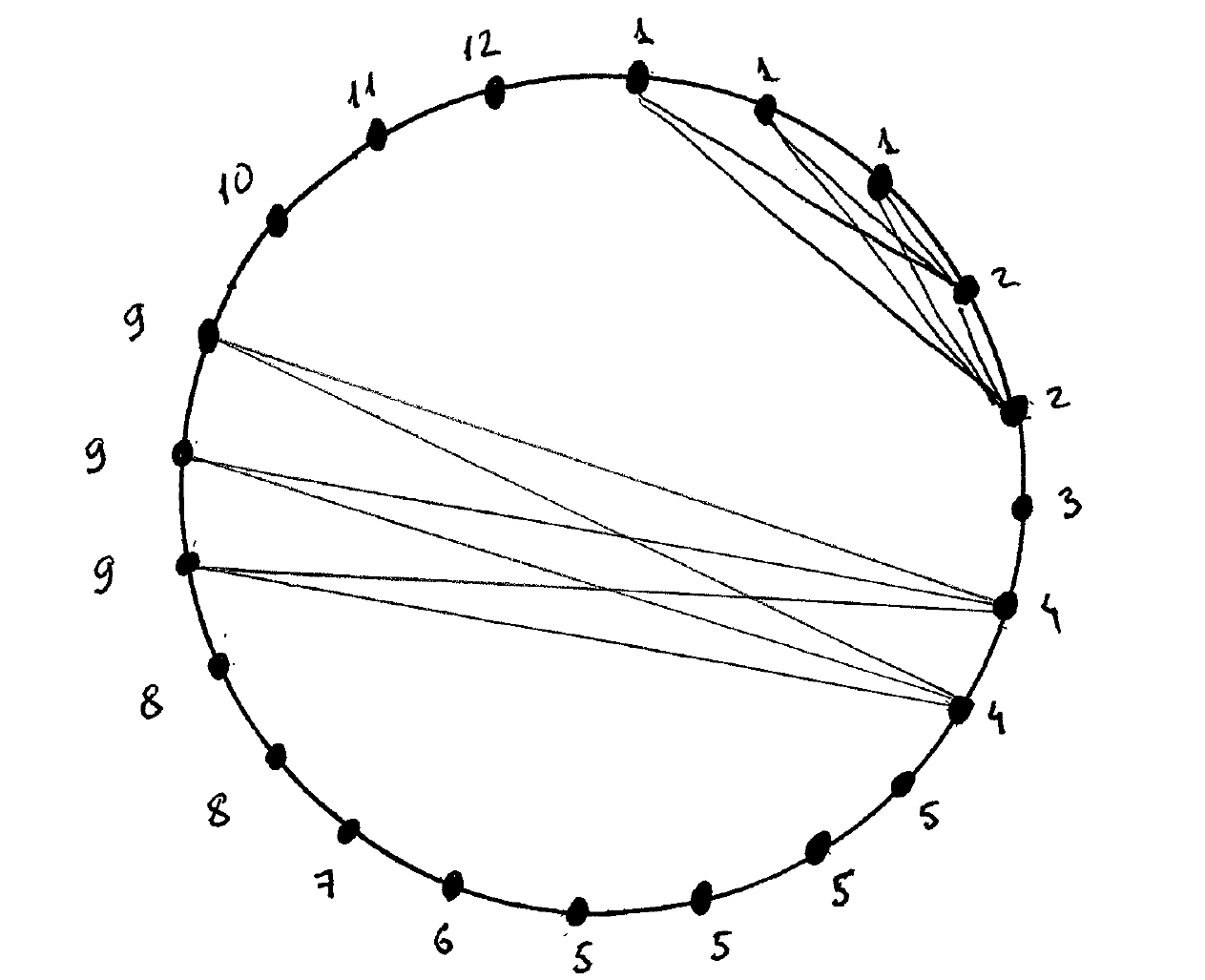}}
\caption{\footnotesize{The standard construction of a weighting in the case $m=11$ and
 $d_1=3, d_2=2, d_3=1, d_4=2, d_5=4, d_6=d_7=1, d_8=2, d_9=3, d_{10}=d_{11}=d_{12}=1$. 
 The indices $1$ and $2$ are joined by $6$ chords of $w_1$-weight $1$ and no chords of length $11$. 
Thus $w(1-2)=6$. The indices $4$ and $9$ are joined by $6$ chords of $w_1$-weight $1$ and $2$ chords of length $11$ and $w_2$-weight $-11$.
 Thus $w(4-9)=-16$.}}\label{F:weighting}
\end{figure}

Next, we note that if all indices from $I$ occur contiguously in $\sigma$, then $I\sqcup J$ is balanced with respect to $\sigma$.
Observe that for any $[C]\in \M_{0,n}$, there exists a cyclic permutation $\sigma \in \S_n$ such that 
all marked points lying on one side of any node of $C$ occur contiguously in $\sigma$. (This can be seen either by induction or 
by examining a planar representation of the dual graph of $C$ as in Figure \ref{F:dual}.) This finishes the proof of (3).
\begin{figure}
\scalebox{1.2}{\includegraphics{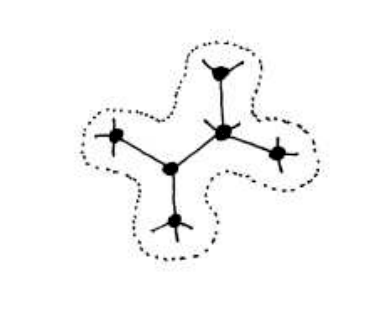}}
\caption{\footnotesize{A planar realization of a dual graph of a $13$-pointed rational curve $C$. 
The half-edges correspond to the marked points. To obtain a cyclic permutation $\sigma$ with respect to which 
all partitions $I\sqcup J$ satisfying $[C]\in \Delta_{I,J}$ are balanced, choose a loop around the graph. The order in which 
the half-edges are encountered as one goes around the loop gives a requisite permutation.}}\label{F:dual}
\end{figure}
By the above, the $w$-flow across $I \sqcup J$ is $\modp{d(I)}{m}\modp{d(J)}{m}$ if and only if 
$I\sqcup J$ is balanced with respect to $\sigma$.
Otherwise, the $w$-flow across $I \sqcup J$ is at least $2m+\modp{d(I)}{m}\modp{d(J)}{m}$
 (indeed, the values of $\sum_{k=1}^m x_k(s-x_k)$ have constant parity).  
To prove (4), it remains to observe that for any partition $I\sqcup J$, there exists $\sigma$ such that $I\sqcup J$ is not balanced 
with respect to $\sigma$. For example, if $I=\{1,\dots,k\}$ and $J=\{k+1,\dots,n\}$, then $\sigma=(k(k+1))(12\cdots n)(k(k+1))$ works. 
\end{proof}

\begin{theorem}\label{T:T1}
$D((d_1,\dots, d_n), m)$ is an effective sum of boundary divisors on 
$\M_{0,n}$ and $\vert D((d_1,\dots, d_n), m)\vert$ is a base-point-free linear system on $\M_{0,n}$. 
Moreover, if $m\nmid d_1\cdots d_n$, then
$D((d_1,\dots, d_n), m)$ separates points of $M_{0,n}$. 
\end{theorem}
\begin{proof}
If $m\mid d_i$, then $D((d_1,\dots,d_n),m)=f^*(D(d_1,\dots,\widehat{d_i},\dots,d_n),m))$, where $f\co \M_{0,n} \ra \M_{0,n-1}$ is the morphism forgetting
the $i^{th}$ marked point. We immediately reduce to the case when $m\nmid d_i$. In this case,
the first claim follows immediately from Lemma \ref{L:main} by applying Proposition \ref{P:standard}(1) and (2). The second claim 
follows from Proposition \ref{P:standard}(3), which says that for any $[C]\in \M_{0,n}$, we can find an effective
combination of the boundary which is linearly equivalent to $D((d_1,\dots,d_n),m)$ and whose support does not contain $[C]$.

Finally, to prove that $D((d_1,\dots, d_n), m)$ separates points of $M_{0,n}$ when $d_i$'s are not divisible by $m$,
it suffices to show that $D((d_1,\dots, d_n), m)$ has a positive degree on any complete irreducible curve $T\subset \M_{0,n}$ meeting
the interior $M_{0,n}$. Let $T$ be such a curve. Since $M_{0,n}$ is affine, there exists a boundary divisor $\Delta_{I,J}$
which meets $T$. By Proposition \ref{P:standard}(4), we can rewrite $D((d_1,\dots, d_n), m)$ as an effective linear
combination of boundary in such a way that the coefficient of $\Delta_{I,J}$ is positive. The claim follows.
\end{proof}

It would be interesting to know if all conformal block divisors on $\M_{0,n}$ are effective combination of boundary.
In view of Theorem \ref{T:T1}, one possible strategy for proving this is to apply the technique of this paper 
to an explicit formula for the divisor classes of the conformal block divisors given by \cite[Proposition 4.3]{mukh}.
(Note that Mukhopadhyay's formula is a direct consequence of \cite{fakh} but presents the divisor class in a form most amenable
to applying Lemma \ref{L:main}.)

\section{Divisor family $\mathcal{D}_2$}\label{S:D2}
In this section, we define and prove nefness for a new family of F-nef divisors on $\M_{0,n}$.

\begin{definition}\label{D:E}
Suppose that $m\geq 3$ is an integer and $m\mid \sum_{i=1}^{n} d_i$. We define
\begin{equation*}
\begin{aligned}
E((d_1,\dots, d_n),m)& :=D((d_1,\dots, d_n),m)+m(\sum_{i:m\mid d_i} \psi_i-\sum_{I\ : \ m\mid d(I)} \Delta_{I,J})\\
=\sum_{i=1}^n \modp{d_i}{m}\modp{m-d_i}{m}\psi_i
&-\sum_{I,J} \modp{d(I)}{m}\modp{d(J)}{m}\Delta_{I,J}
+m(\sum_{i:m\mid d_i} \psi_i-\sum_{I,J\ :  m \mid d(I)} \Delta_{I,J}).
\end{aligned}
\end{equation*}
\end{definition}

The motivation for considering these divisors comes from the following observation. Suppose that $1\leq d_i\leq m-1$.
Then $D((d_1,\dots,d_n),m)$ is a base-point-free divisor on $\M_{0,n}$. It is easy to see that 
the associated morphism $f\co \M_{0,n} \ra X$ contracts the boundary divisor $\Delta_{I,J}$ whenever $m \mid d(I)$.\footnote{
A test family computation in \cite[Corollary 2.6(2)]{fedorchuk-cyclic} 
establishing this for the case $d_1=\cdots=d_n$ applies verbatim in the more general case.} It follows that 
$E((d_1,\dots,d_n),m)$ is of the form $f^*A-E$, where $A$ is a very ample divisor on $X$ and $E$ is an effective combination 
of $f$-exceptional divisors.

\begin{theorem}\label{T:T2}
Suppose $m\geq 3$ and $\{d_i\}_{i=1}^n$ are such that 
$m \mid \sum_{i=1}^n d_i$. Then:
\begin{enumerate}
\item[(a)] $E((d_1,\dots, d_n),m)$ is an effective combination of boundary divisors 
on $\M_{0,n}$.
\item[(b)] $E((d_1,\dots, d_n),m)$ is nef on $\M_{0,n}$.
\end{enumerate}
\end{theorem}

\begin{example}
By taking $n=9$, $d_1=\cdots=d_9=1$, and $m=3$, we obtain the divisor $\Delta_2+\Delta_3+2\Delta_4$ in $\Nef(\M_{0,9})$.
This divisor generates an extremal ray of the nef cone of $\M_{0,9}$ and is not known to come from the conformal block bundles;
see \cite{swinarski}.

It is proved in \cite{fedorchuk-cyclic}, that in the case $d_1=\cdots=d_n$ and $m=3$, the divisor 
$E((d_1,\dots, d_n),m)$ generates an extremal ray of the symmetric nef cone of $\M_{0,n}$. We expect this to be true
more generally whenever $m\geq 5$ is prime and $d_1=\cdots=d_n$. 
\end{example}
 
\begin{proof}[Proof of Theorem \ref{T:T2}]
Replacing $d_i$ by $\modp{d_i}{m}$, we can assume that $0\leq d_i\leq m-1$. Next, we observe that if $d_i=0$, then
\[
E\bigl((d_1,\dots,d_n),m\bigr)=f^*\left(E\bigl((d_1,\dots,\widehat{d_i},\dots,d_n),m\bigr)\right)
+m\psi_i,\] where $f\co \M_{0,n} \ra \M_{0,n-1}$ is the morphism forgetting
the $i^{th}$ marked point. Since $\psi_i$ is well-known to be an effective combination of boundary (see \cite[Lemma 1]{FG}), 
we reduce to the case $1\leq d_i \leq m-1$. Here, Part (a) follows immediately from Lemma \ref{L:main} 
once we establish the existence of a certain weighting on $\Gamma([n])$.
This is achieved in Proposition \ref{P:induction} below.

We proceed to prove Part (b). Since $E((d_1,\dots, d_n),m)$ is an effective combination of the boundary divisors on $\M_{0,n}$ 
by Part (a), it has non-negative degree on any irreducible curve intersecting the interior
$M_{0,n}$. 

Next, observe that $E((d_1,\dots, d_n),m)$ satisfies factorization, that is 
for any boundary divisor $\Delta_{I,J} \subset \M_{0,n}$, we have 
\[
E((d_1,\dots, d_n),m)\vert_{\Delta_{I,J}}=E\bigl((\{d_i\}_{i\in I}, \sum_{j\in J} d_j),m)\bigr)\boxtimes 
E\bigl((\{d_j\}_{j\in J}, \sum_{i\in I} d_i),m)\bigr),
\]
where we use the usual identification $\Delta_{I,J}\simeq \M_{0, I\cup p}\times \M_{0, J\cup q}$.
It follows by a standard argument that $E((d_1,\dots, d_n),m)$ is nef on $\M_{0,n}$.
\end{proof}

In the remainder of this section, we finish the proof of Part (a) of Theorem \ref{T:T2}.
\begin{definition} Suppose $1\leq d_i\leq m-1$ and $m \mid \sum_{i=1}^n d_i$.
Let $w$ be a weighting of a complete graph $\Gamma([n])$. We say that 
\begin{enumerate}
\item[(P1)] $w$ satisfies (P1) with respect to a vertex $i\in \Gamma([n])$ 
if the $w$-flow through $i$ is exactly $d_i(m-d_i)$. 
\item[(P2)] $w$ satisfies (P2) with respect to a proper partition $I\sqcup J=[n]$
if the $w$-flow across $I\sqcup J$ is at least $\modp{d(I)}{m}\modp{d(J)}{m}$. 
\item[(P3)] $w$ satisfies (P3) with respect to a proper $m$-partition 
$I\sqcup J=S$ if the $w$-flow across $I\sqcup J$ is at least $m$.
\end{enumerate}
We say that $w$ satisfies (P1), (P2), and (P3), if it satisfies (P1), (P2), and (P3)
with respect to all vertices, all proper partitions, and all proper $m$-partitions, respectively.
\end{definition}

\begin{prop}\label{P:induction} Let $m\geq 3$.
Suppose $1\leq d_i\leq m-1$ and $m\mid \sum_{i=1}^n d_i$. Then there exists
a weighting $w$ of $\Gamma([n])$ satisfying (P1)-(P3). 
\end{prop}

\begin{proof}
In what follows we say that two partitions $A\sqcup B$ and $C \sqcup D$ are \emph{transverse} if
$\card(A\cap C)\card(A\cap D)\card(B\cap C)\card(B\cap D) >0$. 

Note that the Standard Construction of Proposition \ref{P:standard} produces a weighting satisfying (P1) and (P2).
The most delicate part of the proof is ensuring that (P3) holds. 
We will construct the requisite weighting $w$ by breaking $S$ into smaller pieces using $m$-partitions of $[n]$
and averaging. 

\begin{construction}\label{C:inductive}
Suppose $[n]=S_1\sqcup S_2$ is a proper $m$-partition. By the inductive hypothesis,
there exist weightings $w_1$ and $w_2$ of $\Gamma(S_1)$ and $\Gamma(S_2)$, respectively, 
satisfying (P1)-(P3). These define a weighting $w_{S_1\mid S_2}$ of $\Gamma([n])$ in an obvious way:
\begin{align*}
w_{S_1\mid S_2}(e)=\begin{cases} w_1(e) & \text{if $e\in E(S_1)$}, \\ w_2(e) & \text{if $e\in E(S_2)$}, \\ 0 \ \text{ otherwise} \end{cases} 
\end{align*}
\begin{claim}\label{inductive-claim}
\hfill
\vspace{-0.5pc}
\begin{enumerate}
\item[]
\item $w_{S_1\mid S_2}$ satisfies (P1).
\item $w_{S_1\mid S_2}$ satisfies (P2).
\item The $w_{S_1\mid S_2}$-flow across every $m$-partitions of $[n]$, with the exception 
of $S_1 \sqcup S_2$, is at least $m$. The $w_{S_1\mid S_2}$-flow across $S_1\sqcup S_2$ is $0$. 
The $w_{S_1\mid S_2}$-flow across every $m$-partition $I\sqcup J$ transverse to $S_1\sqcup S_2$ is at least $2m-2$, 
and is at least $2m$ if in addition $m\mid d(S_1\cap I)$. 
\end{enumerate}
\end{claim}
\begin{proof}
The first claim is clear. To prove the second claim, consider a partition $I\sqcup J=[n]$.
Let $r_1=\modp{d(I\cap S_1)}{m}$ and $r_2=\modp{d(I\cap S_2)}{m}$. Without loss of generality, we can 
assume that $r_1+r_2\leq m$. Then the $w_{S_1\mid S_2}$-flow between $I$ and $J$
is at least $r_1(m-r_1)+r_2(m-r_2)\geq (r_1+r_2)(m-r_1-r_2)$, as desired. 

We proceed to prove the third claim. 
Clearly, $w_{S_1\mid S_2}$-flow across $S_1\sqcup S_2$ is $0$.
Let $I\sqcup J$ be another $m$-partition of $[n]$. Suppose first that $I\sqcup J$ is transverse to $S_1\sqcup S_2$.
The $w_{S_1\mid S_2}$-flow across $I\cup J$ is the sum of the $w_1$-flow across $(I\cap S_1) \sqcup (J\cap S_1)$ in $S_1$
and the $w_2$-flow across $(I\cap S_2) \sqcup (J\cap S_2)$ in $S_2$.
Let $r_1=\modp{d(I\cap S_1)}{p}$. Applying the inductive hypothesis we see that
if $r_1=0$, then the resulting $w_{S_1\mid S_2}$-flow is at least $2m$. If $1\leq r_1\leq m-1$,
then the $w_{S_1\mid S_2}$-flow across $I\cup J$ is at least $2r_1(m-r_1) \geq 2m-2$. 

Finally, suppose that $I\sqcup J$ is not transverse to $S_1\sqcup S_2$. Without loss of generality, 
we can assume that $S_1\subsetneq I$. Then the $w_{S_1\mid S_2}$-flow across $I\sqcup J$ 
equals to the $w_2$-flow across $(I\cap S_2)\sqcup (J\cap S_2)$,
which is at least $m$ by the inductive hypothesis 
because $(I\cap S_2)\sqcup (J\cap S_2)$ is a proper $m$-partition of $S_2$. 
\end{proof}
\end{construction}
Using Construction \ref{C:inductive}, we proceed to construct the weighting $w$ using averaging and induction on $\sum_{i=1}^n d_i$.
The case of $\sum_{i=1}^n d_i=m$ follows from Proposition \ref{P:standard} because there are no proper $m$-partitions.
Suppose $\sum_{i=1}^n d_i=sm$, where $s\geq 2$. 

\emph{Case 1: There is at most one proper $m$-partition $I\mid J$ of $[n]$.} 
In this case, the claim follows by Proposition \ref{P:standard}(4) because we can arrange 
the flow across the unique $m$-partition to be at least $2m$.

\emph{Case 2: There are exactly two distinct $m$-partitions of $[n]$.} Call them $A\sqcup B$ and $C\sqcup D$. By the assumption
these must be transverse (otherwise, there would exist at least $3$ distinct $m$-partitions). 
By Proposition \ref{P:standard}(4), there exists a weighting $w_1$ of $\Gamma([n])$ 
satisfying (P1)-(P2) and $w_1(A\mid B)\geq 2m$. Similarly, there exists a weighting $w_2$ of $\Gamma([n])$
satisfying (P1)-(P2) and $w_2(C\mid D)\geq 2m$. It follows that $w:=(w_1+w_2)/2$ is a weighting of $\Gamma([n])$ satisfying (P1)-(P3).

{\em Case 3: $[n]$ is not a disjoint union of three non-trivial $m$-divisible subsets and there are $k\geq 3$ 
distinct $m$-partitions of $[n]$.} Let $\{ A_i \sqcup B_i\}_{i=1}^k$ be all $m$-partitions of $[n]$. Let $w_i:=w_{A_i\mid B_i}$ 
as constructed in Construction \ref{C:inductive}. Then $w:=(\sum_{i=1}^k w_i)/k$ is a weighting of $\Gamma([n])$ satisfying (P1)-(P3).
Indeed, (P1)-(P2) clearly hold. Furthermore, 
for any $m$-partition $A_i\sqcup B_i$ and $j\neq i$, we have $w_j(A_i\sqcup B_i) \geq 2m-2$ by
Claim \ref{inductive-claim}(3). It follows that
 \[
w(I\mid J)\geq (k-1)(2m-2)/k \geq m,\] 
if $m\geq 4$ or if $m=3$ and $k\geq 4$. (If $m=3$, it is easy to see that $k\geq 4$.)

\emph{Case 3: $[n]$ is a disjoint union of four non-trivial $m$-divisible subsets.} 
Let $[n]=A\cup B\cup C\cup D$, where $m\mid d(A),d(B),d(C),d(D)$. 
Using Construction \ref{C:inductive}, we obtain 
the following three weightings of $[n]$:
\[
w_1:=w_{(A\cup B) \mid (C\cup D)}, \quad w_2:=w_{(A\cup C) \mid (B\cup D)}, \quad w_3:=w_{(A\cup D) \mid (B\cup C)}.
\]
Set $w:=(w_1+w_2+w_3)/3$. Then 
\[w((A\cup B) \mid (C\cup D)) \geq (0+2m+2m)/3 > m.\]
Hence $w$ satisfies (P3) with respect to $(A\cup B) \sqcup (C\cup D)$.
Similarly, it is easy 
to see from the construction that $w$ satisfies (P3) with respect to all proper $m$-partitions.  

{\em Case 4: $[n]$ is a disjoint union of three non-trivial $m$-divisible subset but not a disjoint union of four non-trivial $m$-divisible subsets.}
The case of $m=3$ is straightforward and so we assume $m\geq 4$ in what follows. 
Let $[n]=S_1\cup S_2 \cup S_3$, where $m\mid d(S_i)$. 

Suppose first that $S_i\sqcup S_j$ is the unique $m$-partition of $(S_i\cup S_j)$
for all $i\neq j$. Let $w_{12}:=w_{(S_1\cup S_2)\mid S_3}$ be the weighting of $\Gamma([n])$ 
from Construction \ref{C:inductive}.
Since $S_1\cup S_2$ has a unique $m$-partition, we can arrange 
the $w_{12}$-flow across $S_1\sqcup S_2$ in $S_1\cup S_2$
to be at least $2m$ by Proposition \ref{P:standard}(4). Define analogously $w_{13}$ and $w_{23}$. Then 
\[
w:=(w_{12}+w_{13}+w_{23})/3
\]
is a weighting of $\Gamma([n])$ satisfying (P1)-(P3).

Finally, without loss of generality, suppose that $S_1\cup S_2$ 
has an $m$-partition $A\sqcup B$ distinct from $S_1\sqcup S_2$.
Then $A\sqcup B$ must be transverse to $S_1\sqcup S_2$. 
The average of the following weightings of $\Gamma([n])$
constructed using Construction \ref{C:inductive}:
\[
w_{S_1\mid (S_2\cup S_3)}, \ w_{S_2\mid (S_1\cup S_3)}, \ w_{A\mid (B\cup S_3)}, \ w_{B\mid (A\cup S_3)}
\]
satisfies (P1)-(P3). 
Indeed, we must verify that the flow across $S_1 \sqcup (S_2\cup S_3)$ is at least $m$, the other cases
being analogous or easier. 
By construction, the flow across $S_1 \sqcup (S_2\cup S_3)$ is at least $(0+m+(2m-2)+(2m-2))/4\geq m$. 
\end{proof}

\subsection*{Acknowledgements.} 
The author was partially supported by NSF grant DMS-1259226. We thank 
David Swinarski for helpful discussions related to this work and for help with experimental computations 
performed using his Macaulay 2 package {\tt ConfBlocks}. 
We also thank Anand Deopurkar for hospitality and stimulating discussions.

\bibliographystyle{alpha}
\bibliography{references20}
\end{document}